\documentclass[11pt,onecolumn]{IEEEtran}

\usepackage{tikz}
\usepackage{nicefrac}
\usepackage{graphicx}
\usepackage{amsmath}
\usepackage{amsfonts}
\usepackage{amssymb}
\usepackage{mathrsfs}
\usepackage{fix2col}
\usepackage{latexsym}
\usepackage[dvips]{epsfig}
\usepackage[font=footnotesize]{subfig}
\usepackage{hyperref}
\usepackage{graphics}
\usepackage{epstopdf}
\usepackage{enumerate}
\usepackage{epstopdf}
\usepackage{amsthm}
\usepackage{arydshln}
\usepackage{mathdots}
\usepackage{wrapfig}
\usepackage[draft]{commands}
\usepackage{mathtools}
\DeclareMathAlphabet{\mathbbm}{U}{bbm}{m}{n}
\usepackage{algorithm,algcompatible}
\usepackage[normalem]{ulem} 

\newcommand{\vertiii}[1]{{\left\vert\kern-0.25ex\left\vert\kern-0.25ex\left\vert #1 
		\right\vert\kern-0.25ex\right\vert\kern-0.25ex\right\vert}}

\makeatletter
\newcommand{\opnorm}{\@ifstar\@opnorms\@opnorm}
\newcommand{\@opnorms}[1]{%
	\left|\mkern-1.5mu\left|\mkern-1.5mu\left|
	#1
	\right|\mkern-1.5mu\right|\mkern-1.5mu\right|
}
\newcommand{\@opnorm}[2][]{%
	\mathopen{#1|\mkern-1.5mu#1|\mkern-1.5mu#1|}
	#2
	\mathclose{#1|\mkern-1.5mu#1|\mkern-1.5mu#1|}
}
\makeatother

\title{A Tutorial on Concentration Bounds for System Identification}
	\author{Nikolai Matni, Stephen Tu}

\begin{document}
\bstctlcite{IEEEexample:BSTcontrol}

\maketitle

\begin{abstract}
We provide a brief tutorial on the use of concentration inequalities as they apply to system identification of state-space parameters of linear time invariant systems, with a focus on the fully observed setting. We draw upon tools from the theories of large-deviations and self-normalized martingales, and provide both data-dependent and independent bounds on the learning rate.
\end{abstract}

\section{Introduction}
\label{sec:intro}
A key feature in modern reinforcement learning is the ability to provide high-probability guarantees on the finite-data/time behavior of an algorithm acting on a system.  The enabling technical tools used in providing such guarantees are concentration of measure results, which should be interpreted as quantitative versions of the strong law of large numbers.  This paper provides a brief introduction to such tools, as motivated by the identification of linear-time-invariant (LTI) systems.

In particular, we focus on the identifying the parameters $(A,B)$ of the LTI system
\begin{equation}
x_{t+1} = Ax_t + Bu_t + w_t,
\label{eq:sys-intro}
\end{equation}
assuming \emph{perfect} state measurements.  This is in some sense the simplest possible system identification problem, making it the perfect case study for such a tutorial.  Our companion paper \cite{extended} shows how the results derived in this paper can then be integrated into self-tuning and adaptive control policies with finite-data guarantees.  We also refer the reader to Section II of \cite{extended} for an in-depth and comprehensive literature review of classical and contemporary results in system identification.  Finally, we note that most of the results we present below are not the sharpest available in the literature, but are rather chosen for the pedagogical value.

The paper is structured as follows: in Section \ref{sec:scalar}, we study the simplified setting when system \eqref{eq:sys-intro} is defined for a scalar state $x$, and data is drawn from independent experiments.  Section \ref{sec:vector} extends these ideas to the vector valued settings.  In Section \ref{sec:single} we study the performance of an estimator using all data from a single trajectory -- this is significantly more challenging as all covariates are strongly correlated.  Finally, in Section \ref{sec:data-dependent}, we provide data-dependent bounds that can be used in practical algorithms. 
 

\section{Scalar Random Variables}
\label{sec:scalar}

Consider the scalar dynamical system
\begin{equation}
x_{t+1} = ax_t + u_t + w_t,
\label{eq:scalar-system}
\end{equation}
for $w_t\iid \Normal(0,\sigma_w^2)$, and $a\in \R$ an unknown parameter.
Our goal is to estimate $a$, and to do so we inject excitatory Gaussian noise via $u_t \iid \Normal (0,\sigma_u^2)$.  We run $N$ experiments over a horizon of $T+1$ time-steps, and then solve for our estimate $\hat{a}$ via the least-squares problem
\begin{equation}\begin{array}{rcl}
\hat{a} &=& \arg\min_a \sum_{i=1}^N (x_{T+1}^{(i)} - ax_{T}^{(i)}- u_T^{(i)})^2\\
&=& a + \frac{\sum_{i=1}^{N}x_T^{(i)} w_T^{(i)}}{\sum_{i=1}^{N} (x^{(i)}_T)^2} =: a + e_N.
\end{array}
\label{eq:scalar-ols}
\end{equation}
Notice that we are using only the last two data-points from each trial -- this simplifies the analysis of the error term $e_N$ greatly as each of the summands in the numerator and denominator are now i.i.d. random variables.  Our goal is to provide high-probability bounds on this error term, and return to the single trajectory estimator later in the paper.

\subsubsection{Bounded Random Variables}
To build some intuition we begin by studying the behavior of almost surely (a.s.) bounded random variables.  In particular, let $\{X_i\}_{i=1}^N$ be drawn i.i.d. from a distribution $p$, and let $X_i \in [a,b]$ a.s. for all $i$.  Our goal is to quantify, with high-probability, the gap between the empirical and true means, i.e., to find a bound on
\begin{equation}
\left|\frac{1}{N}\sum_{i=1}^N X_i - \E X_1 \right|
\end{equation}
that holds with high-probability.

When working with bounded random variables \emph{McDiarmid's inequality} is a very powerful tool for establishing concentration of measure.

\begin{theorem}{McDiarmid's Inequality}
Let $X_i \in \mathcal{X}$ for $i = 1, \dots, N$ be drawn independently, and let $F:\X^n \to \R$ satisfies.  If, for all $i=1,\dots,N$, and all $x_1,\dots,x_N,x_i' \in \X$ it holds that
\begin{align}
\displaystyle\sup_{x_1,\dots,x_N,x_i'}\left|F(x_1,\dots,x_N) - F(x_1,\dots, x_{i-1}, x_i',x_{i+1},\dots, x_N)\right| \leq c_i,
\label{eq:bounded}
\end{align}
then we have that
\begin{equation}
\Prob{F(x_1,\dots,x_N) - \E\left[F(x_1,\dots,x_N)\right] \geq t} \leq \exp{-\frac{2t^2}{\sum_{i=1}^nc_i^2}}.
\label{eq:mcdiarmid_prob}
\end{equation}
\label{thm:McDiarmid}
\end{theorem}

From Theorem \ref{thm:McDiarmid}, one can easily derive the \emph{Hoeffding's inequality for bounded random variables}.

\begin{coro}[Hoeffding's inequality for bounded random variables]
Let $\{X_i\}_{i=1}^N \iid p^N$ be such that $X_i \in [a,b]$ a.s..  Then
\begin{equation}
\Prob{\frac{1}{N}\sum_{i=1}^N X_i - \E X_1 \geq t} \leq \exp{\frac{-2Nt^2}{(b-a)^2}}.
\label{eq:hoeff-bound}
\end{equation}
\label{coro:Hoeffding-Bounded}
\end{coro}
\begin{proof}
Set $F(x_1,\dots,x_N) = \frac{1}{N}\sum_{i=1}^N x_i$ and notice that it satisfies the boundedness condition \eqref{eq:bounded} with $c_i \equiv (b-a)/N$ for all $i$.
\end{proof}

\begin{example}[Probability Estimation]
Let $\{X_i\}_{i=1}^N \iid p^N$ be random vectors in $\mathcal{X}$, and let $\Omega \subseteq \mathcal{X}$ be some set.  Let
\begin{equation}
\hat{P}_N = \frac{1}{N} \sum_{i=1}^N \Ind{x \in \Omega},
\end{equation}
and notice that $\E\hat{P}_N = \Prob{x \in \Omega}$.  As $\Ind{x \in \Omega} \in \{0,1\}$ for all $x$, it follows by equation \eqref{eq:hoeff-bound} that 
\[
\Prob{\hat{P}_N - \Prob{x \in \Omega} \geq t} \leq \exp{-2Nt^2}.
\]
We can obtain a similar bound on the probability of the event $\left\{\hat P_N - \Prob{x \in \Omega} \leq -t\right\}$ occurring: it then follows by union bounding over these two events that
\[
\Prob{\left|\hat{P}_N - \Prob{x \in \Omega} \right|\geq t} \leq 2\exp{-2Nt^2}.
\]
\end{example}

Thus we have seen that in the case of a.s. bounded random variables, concentration of measure does indeed occur.
We will now see that similar concentration occurs for random variables drawn from distributions with sufficiently rapidly decaying tails.  

\subsubsection{Sub-Gaussian Random Variables}
We begin by recalling the Chernoff bound, which states that for a random variable $X$ with mean $\E X$, and moment generating function (MGF) $\E\left[e^{\lambda(X-\E X)}\right]$ defined for all $|\lambda|\leq b$, it holds that
\begin{equation}
\log\Prob{X - \E X \geq t} \leq \inf_{\lambda \in [0,b]}\left[\log\E\left[e^{\lambda(X-\E X)}\right] - \lambda t \right].
\label{eq:chernoff}
\end{equation}

We now turn our attention to Gaussian random variables, and recall that for $X \sim \Normal (\mu,\sigma^2)$, we have that $\E\left[e^{\lambda(X-\mu)}\right]= \exp{\frac{\sigma^2 \lambda^2}{2}}$ for all $\lambda \in \R$.  Substituting this into the Chernoff bound \eqref{eq:chernoff} and solving for $\lambda^\star = t/\sigma^2$, we immediately obtain
\begin{equation}
\Prob{X-\mu \geq t} \leq \exp{\frac{-t^2}{2\sigma^2}}.
\label{eq:gauss-conc}
\end{equation}

Recalling that if $X_1 \sim \Normal(\mu_1, \sigma_1^2)$ and $X_2 \sim \Normal(\mu_2,\sigma_2^2)$ then $aX_1+bX_2 \sim \Normal(a\mu_1 + b\mu_2, a^2\sigma_1^2 + b^2 \sigma_2^2)$, it follows immediately that for $X_i \iid \Normal(\mu,\sigma^2)$, it holds that
\begin{equation}
\Prob{\frac{1}{N}\sum_{i=1}^N X_i - \mu \geq t} \leq \exp{\frac{-Nt^2}{2\sigma^2}}.
\label{eq:gauss-mean-prob}
\end{equation}
Once again, a similar bound can be obtained on the probability of event $\left\{\frac{1}{N}\sum_{i=1}^N X_i - \mu \leq -t \right\}$ occurring: it then follows by union bounding over these two events that 
\begin{equation}
\Prob{\left|\frac{1}{N}\sum_{i=1}^N X_i - \mu \right|\geq t} \leq 2\exp{\frac{-Nt^2}{2\sigma^2}}
\label{eq:gauss-mean-deviation}
\end{equation}

We now generalize these results to random variables with MGFs dominated by that of a Gaussian random variable.

\begin{definition}[Sub-Gaussian Random Variable]
A random variable $X$ with mean $\E X$ is sub-Gaussian if there exists a positive number $\sigma^2$ such that
\begin{equation}
\mgf \leq \exp{\frac{\lambda^2\sigma^2}{2}} \ \forall \lambda \in \R.
\label{eq:sub-gauss}
\end{equation}
\label{def:sub-gauss}
\end{definition}
An example of random variables that are sub-Gaussian but not Gaussian are bounded random variables -- it can be shown that a random variable $X$ taking values in $[a,b]$ almost surely satisfies equation \eqref{eq:sub-gauss} with parameter $\sigma^2 = (b-a)^2/4$.

Further, from this definition, it follows immediately that from the Chernoff bound that all sub-Gaussian random variables satisfy the concentration bound \eqref{eq:gauss-conc}.  One can also check that if $X_1$ and $X_2$ are sub-Gaussian with parameters $\sigma_1^2$ and $\sigma_2^2$, then $X_1 + X_2$ is sub-Gaussian with parameter $\sigma_1^2 + \sigma_2^2$, from which we immediately obtain Hoeffding's Inequality.

\begin{theorem}[Hoeffding's Inequality]
Let $\{X_i\}_{i=1}^N$ be iid sub-Gaussian random variables with parameter $\sigma^2$.  Then
\begin{equation}
\Prob{\frac{1}{N}\sum_{i=1}^N X_i - \E X_1 \geq t} \leq \exp{\frac{-Nt^2}{2\sigma^2}}.
\label{eq:hoeff-subG}
\end{equation}
\label{thm:Hoeffding-subG}
\end{theorem}
\paragraph*{An aside on probability inversion and two sided bounds}  Rather than statements about the probability of large deviations, as in bound \eqref{eq:hoeff-subG}, we are often interested in the probability that a random variable concentrates near its mean.  To do so, we employ probability inversion:  if we are willing to tolerate a large deviation occurring with probability at most $\delta$, one may invert bound \eqref{eq:hoeff-subG} by setting $\delta = $ RHS of \eqref{eq:hoeff-subG} and solving for $t$.  This allows us to certify that with probability at least $1-\delta$ that
\begin{equation}
\frac{1}{N}\sum_{i=1}^N X_i \leq  \E X_1 + \sqrt{\frac{2\sigma^2\log(1/\delta)}{N}}.
\end{equation}
Applying the same reasoning to the event $\{X - \E X \leq -t\}$ yields a similar bound, from which it follows, by the union bound, that with probability at least $1-2\delta$ that 
\begin{equation}
\left|\frac{1}{N}\sum_{i=1}^N X_i - \E X_1 \right|\leq \sqrt{\frac{2\sigma^2\log(1/\delta)}{N}}.
\label{eq:two-sided-bound}
\end{equation}


\subsubsection{Sub-Exponential Random Variables}
Revisiting the error term defined in \eqref{eq:scalar-ols}, we see that we still do not have the requisite tools to perform the desired analysis.

\begin{example}[Products of Gaussians are not Sub-Gaussian]
Motivated by the error term in \eqref{eq:scalar-ols}, we compute the MGFs for $X^2$ and $XW$, where $X,W\iid \Normal(0,1)$.  Direct computation of the resulting integrals show that
\begin{equation}
\begin{array}{rcl}
\E\left[e^{\lambda(X^2-1)}\right] &=& \frac{e^{-\lambda}}{1-2\lambda}\text{ if $\lambda < 1/2$,} \\
\E\left[e^{\lambda(XW)}\right] &=& \frac{1}{\sqrt{\pi(1-\lambda^2)}}\text{ if $|\lambda|<1$}.
\end{array}
\label{eq:mgfs-for-scalar}
\end{equation}
These random variables are clearly not sub-Gaussian, as their MGFs do not exist for all $\lambda \in \R$.  However, notice that they can be bounded by the MGF of a Gaussian random variable in a neighborhood of the origin.  In particular we have that
\begin{equation*}
\begin{array}{rcll}
\E\left[e^{\lambda(X^2-1)}\right] &=& \frac{e^{-\lambda}}{1-2\lambda} &\leq \exp{\frac{4\lambda^2}{2}}\text{ $\forall |\lambda| < 1/4$} \\
\E\left[e^{\lambda(XW)}\right] &=& \frac{1}{\sqrt{\pi(1-\lambda^2)}} & \leq \exp{\frac{2\lambda^2}{2}}\text{ $\forall |\lambda|<1/\sqrt{2}$}.
\end{array}
\label{eq:subexp-example}
\end{equation*}
The first inequality follows from some calculus, and the second by leveraging that $-\log(1-x)\leq x(1-x)^{-1}$ for $0 \leq x <1$.
\label{ex:chi-squared}
\end{example}

We now show that MGFs exhibiting behavior as above also concentrate.

\begin{definition}
A random variable $X$ with mean $\E X$ is sub-exponential with parameters $(\nu^2, \alpha)$ if
\begin{equation}
\E\left[e^{\lambda(X-\E X)}\right] \leq \exp{\frac{\nu^2\lambda^2}{2}} \ \forall |\lambda| \leq \frac{1}{\alpha}.
\end{equation}
\label{def:subexp}
\end{definition}

Example \ref{ex:chi-squared} therefore demonstrated that for $X,W \iid \Normal(0,1)$, $X^2$ is sub-exponential with parameters $(4,4)$, and $XW$ is sub-exponential with parameters $(2,\sqrt{2})$. 

We now state without proof the tail bound enjoyed by sub-exponential random variables, which follows from a more involved Chernoff type argument (see Ch. 2 of \cite{wainwright2019high}).  Specifically, if $X$ is sub-exponential with parameters $(\nu^2,\alpha)$, then
\begin{equation}
\Prob{X - \E X \geq t} \leq \begin{cases} \exp{\frac{-t^2}{2\nu^2}}\text{ if $0\leq t \leq \frac{\nu^2}{\alpha}$} \\
\exp{\frac{-t}{2\alpha}}\text{ if $t > \frac{\nu^2}{\alpha}.$}\end{cases}
\label{eq:sub-exp-bound}
\end{equation}
Thus we see that for sufficiently small deviations $0\leq t \leq \nu^2/\alpha$, sub-exponential random variables exhibit sub-Gaussian concentration -- indeed, informally, one may view sub-Gaussian random variables as the limit of a sub-exponential random variable with $\alpha \to 0$.  Finally, we note that we can show that for $X_1$ and $X_2$ sub-exponential random variables with parameters $(\nu_i^2,\alpha_i^2)$, we have that $X_1 + X_2$ is a sub-exponential random variable with parameters $(\nu_1^2 + \nu_2^2, \max\{\alpha_1, \alpha_2\})$.  

We now return to our motivating example and analyze the error term \eqref{eq:scalar-ols}.  First, we observe that
\begin{equation} 
 x_T^{(i)} \iid \Normal(0,(\sigma_w^2 + \sigma_u^2) \sum_{t=0}^T a^{2t}).
 \label{eq:scalar-iid-xT}
 \end{equation}
In what follows, we let $\sigma_x^2:= (\sigma_w^2 + \sigma_u^2) \sum_{t=0}^T a^{2t}$, which we recognize as the (variance weighted) finite-time controllability Gramian of the scalar system \eqref{eq:scalar-system}. 

\begin{theorem}
Consider the least squares estimator \eqref{eq:scalar-ols}.  Fix a failure probability $\delta \in (0,1]$, and assume that $N\geq 32 \log(2/\delta)$.  Then with probability at least $1-\delta$, we have that
\begin{equation}
|e_N| \leq 4\frac{\sigma_w}{\sigma_x}\sqrt{\frac{\log(4/\delta)}{N}}.
\end{equation}
\label{thm:iid-scalar-error}
\end{theorem}

This theorem follows immediately by invoking the next two propositions with failure probability $\delta/2$ and union bounding.

\begin{proposition}
Fix $\delta \in (0,1]$, and let $N \geq 32 \log(1/\delta)$.  Then with probability at least $1-\delta$
\begin{equation}
\sum_{i=1}^N (x_T^{(i)})^2 \geq \sigma_x^2 \frac{N}{2}.
\label{eq:scalar-iid-lb}
\end{equation}
\label{prop:scalar-iid-lb}
\end{proposition}
\begin{proof}
From \eqref{eq:scalar-iid-xT}, we have that $x_T/\sigma_x \sim \Normal(0,1)$.  Thus from Example \ref{ex:chi-squared}, $x_T^2/\sigma_x^2$ is sub-exponential with parameters $(4,4)$, and 
$\sum_{i=1}^N \left(\nicefrac{x_T^{(i)}}{\sigma_x}\right)^2$
is sub-exponential with parameters $(4N,4)$.  From the tail bound \eqref{eq:sub-exp-bound}, we see that
\begin{align*}
&\Prob{\sigma_x^2\sum_{i=1}^N \left(\frac{x_T^{(i)}}{\sigma_x}\right)^2 - N\sigma_x^2 \leq -t } = \Prob{\sum_{i=1}^N \left(\frac{x_T^{(i)}}{\sigma_x}\right)^2 - N \leq \frac{-t}{\sigma_x^2}} \leq \exp{\frac{-t^2}{8N\sigma_x^4}},
\end{align*}
for all $t \leq N\sigma_x^2$.  Inverting this bound to solve for a failure probability of $\delta$, we see that $t = \sigma_x^2\sqrt{8N\log(1/\delta)} \leq N\sigma_x^2$, where the inequality follows from out assumed lower bound on $N$.  We therefore have, with probability at least $1-\delta$, that
\begin{align*}
\sum_{i=1}^N (x_T^{(i)})^2 \geq \sigma_x^2(N-\sqrt{8N\log(1/\delta)}) \geq \sigma_x^2\frac{N}{2},
\end{align*}
where the final inequality follows from the assumed lower bound on $N$.
\end{proof}

\begin{proposition}
Fix $\delta \in (0,1]$, and let $N \geq \tfrac{1}{2} \log(2/\delta)$.  Then with probability at least $1-\delta$
\begin{equation}
\left|\sum_{i=1}^N x_T^{(i)}w_T^{(i)}\right| \leq 2\sigma_x\sigma_w \sqrt{N\log(2/\delta)}.
\label{eq:scalar-iid-ub}
\end{equation}
\label{prop:scalar-iid-ub}
\end{proposition}
\begin{proof}
By a similar argument as the previous proof, we have that
$\sum_{i=1}^N \nicefrac{x_T^{(i)}w_T^{(i)}}{\sigma_x\sigma_w}$
is sub-exponential with parameters $(4N,\sqrt{2})$, from which it follows that
\begin{align*}
\Prob{ \left|\sum_{i=1}^N x_T^{(i)}w_T^{(i)}\right| \geq t} &= \Prob{\left|\sum_{i=1}^N \frac{x_T^{(i)}w_T^{(i)}}{\sigma_x\sigma_w}\right| \geq \frac{t}{\sigma_x\sigma_w}} \leq 2\exp{\frac{-t^2}{4N\sigma_x^2\sigma_w^2}},
\end{align*}
if $t\leq 2\sqrt{2}N\sigma_x\sigma_w$.  Inverting with probability failure $\delta$, we obtain $t = 2\sigma_x\sigma_w\sqrt{N\log(2/\delta)} \leq 2\sqrt{2}N\sigma_x\sigma_w$, where the final inequality holds by the assumed lower bound on $N$.  Thus, with probability at least $1-\delta$, \eqref{eq:scalar-iid-ub} holds.
\end{proof}

\section{Vector Valued Random Variables}
\label{sec:vector}

Consider a linear dynamical system described by
\begin{equation}
x_{t+1} = Ax_t + Bu_t + w_t,
\end{equation}
where $x_t,w_t \in \R^{n_x}$, $u_t \in \R^{n_u}$, and $w_t \iid \Normal(0,\sigma_w^2 I)$.  Our goal is to identify the matrices $(A,B)$,
 and to do so we inject excitatory Gaussian noise via $u_t \iid \Normal (0,\sigma_u^2 I)$.  As before, we run $N$ experiments over a horizon of $T+1$ time-steps.  Letting
 \begin{equation}
X_N := \begin{bmatrix}   (x_{T+1}^{(1)})^\top \\ \vdots \\ (x_{T+1}^{(N)})^\top\end{bmatrix},\, Z_N := \begin{bmatrix} (x_T^{(1)})^\top,\,  (u_T^{(1)})^\top\\ \vdots \\ (x_T^{(N)})^\top,\,  (u_T^{(N)})^\top\end{bmatrix},\,  W_N := \begin{bmatrix} (w_{T}^{(1)})^\top \\ \vdots \\(w_{T}^{(N)})^\top\end{bmatrix},
  \end{equation}
 we then solve for our estimates $(\Ahat,\Bhat)$ via the least-squares problem:
\begin{equation}\begin{array}{rcl}
\begin{bmatrix} \Ahat & \Bhat \end{bmatrix}^\top &=& \arg\min_{(A,B)} \sum_{i=1}^N \twonorm{x_{T+1}^{(i)} - Ax_{T}^{(i)}- Bu_T^{(i)}}^2\\
&=& \arg\min_{(A,B)} \bignorm{X_N - Z_N\begin{bmatrix} A & B \end{bmatrix}^\top}_F^2\\
&=& \begin{bmatrix} A & B \end{bmatrix}^\top + (Z_N^\top Z_N)^{-1}Z_N^\top W_N.
\end{array}
\label{eq:matrix-ls}
\end{equation}
Notice that
\begin{equation}
Z_N^\top Z_N = \sum_{i=1}^N \begin{bmatrix} x_T^{(i)} \\ u_T^{(i)} \end{bmatrix}\begin{bmatrix} x_T^{(i)} \\ u_T^{(i)} \end{bmatrix}^\top,\, Z_N^\top W_N =  \sum_{i=1}^N\begin{bmatrix} x_T^{(i)} \\ u_T^{(i)} \end{bmatrix}(w_T^{(i)})^\top,
\end{equation}
where one can easily verify that
\begin{equation}
\begin{bmatrix} x_T^{(i)} \\ u_T^{(i)} \end{bmatrix}\iid \Normal\left(0,
\begin{bmatrix}
\sigma_u^2 \Lambda_C(A,B,T) + \sigma_w^2\Lambda_C(A,I,T)& 0 \\
0 & \sigma_u^2 I_{n_u}
\end{bmatrix}
\right),
\label{eq:covariance}
\end{equation}
with
\begin{equation}
\Lambda_C(A,B,T) = \sum_{t=0}^T A^tBB^\top (A^\top)^t,
\end{equation}
the $T$-time-step controllability Gramian of $(A,B)$.  To lighten notation we let $\Sigma_x$ denote the $(1,1)$ block of the above covariance, i.e., 
\[
\Sigma_x :=\sigma_u^2 \Lambda_C(A,B,T) + \sigma_w^2\Lambda_C(A,I,T).
\]

Our objective is to derive high-probability bounds on the spectral norm of the error terms
\begin{equation}
\begin{array}{rcrcl}
E_A&:=& (\Ahat - A)^\top &=&  \begin{bmatrix}I_{n_x}  & 0_{n_x\times n_u}\end{bmatrix} (Z_N^\top Z_N)^{-1}Z_N^\top W_N \\
E_B&:=&(\Bhat - B)^\top &=&  \begin{bmatrix}0_{n_u\times n_x}  & I_{n_u}\end{bmatrix} (Z_N^\top Z_N)^{-1}Z_N^\top W_N.
\end{array}
\label{eq:iid-Matrix-error}
\end{equation}
Define $Q_A = \begin{bmatrix}I_{n_x} & 0_{n_x\times n_u}\end{bmatrix}$.  Then $(\Ahat - A)^\top = Q_A E_N$, and 
\begin{align}
\twonorm{\Ahat-A} & = \twonorm{Q_A(Z_N^\top Z_N)^{-1}Z_N^\top W_N} \nonumber\\
& \dist \twonorm{Q_A(\Sigma^{1/2}Y_N^\top Y_N \Sigma^{1/2})^{-1}\Sigma^{1/2}Y_N^\top W_N} \nonumber\\
& = \twonorm{Q_A\Sigma^{-1/2}(Y_N^\top Y_N)^{-1}Y_N^\top W_N} \nonumber\\
&\leq \twonorm{\Sigma_x^{-1/2}}\frac{\twonorm{Y_N^\top W_N}}{\lambda_{\min}(Y_N^\top Y_N)} \nonumber\\
& = \lambda_{\min}^{-1/2}(\Sigma_x)\frac{\twonorm{Y_N^\top W_N}}{\lambda_{\min}(Y_N^\top Y_N)}, \label{eq:error_A}
\end{align}
where $Y_N := [y_i^\top]_{i=1}^N$, with $y_i \iid \Normal(0,I_{n_x+n_u})$.  A similar argument shows that 
\begin{align}
\twonorm{\Bhat - B} & \leq \frac{1}{\sigma_u}\frac{\twonorm{Y_N^\top W_N}}{\lambda_{\min}(Y_N^\top Y_N)},
\label{eq:error_B}
\end{align}
reducing our task to finding (i) an upper bound on the norm of the cross term $\sum_i y_i w_i^\top$, and (ii) a lower bound on the minimum eigenvalue of the empirical Gramian matrix $\sum_i y_i y_i^\top$.  These can be obtained using tail bounds for sub-gaussian and sub-exponential random variables, as formalized in the following propositions.
%
%

Now, recall that for a matrix $M = \sum_{i=1}^N x_iw_i^\top$
\begin{equation}
\twonorm{M} = \sup_{u\in \Sp^{n-1}, v\in \Sp^{m-1}}\sum_{i=1}^N (u^\top x_i)(w_i^\top v).
\end{equation}
We begin our analysis by fixing a $(u,v) \in \Sp^{n-1}\times\Sp^{m-1}$, and notice that each $(u^\top x_i)$ and $(w_i^\top v)$ are sub-Gaussian random variables with parameter $\sigma^2 = 1$ if the components $x_{i,j}$, $j=1,\dots,n$, and $w_{i,k}$, $k=1,\dots,m$, are themselves sub-Gaussian random variables with parameters $\sigma^2=1$.  From Example \ref{ex:chi-squared}, we conclude that $(u^\top x_i)(w_i^\top v)$ is a zero-mean sub-exponential random variabled with parameters $(2,\sqrt{2})$.  It then follows immediately from a one-sided version of Proposition \ref{prop:scalar-iid-ub}, that for a fixed $(u,v) \in \Sp^{n-1}\times\Sp^{m-1}$, if $N \geq \tfrac{1}{2}\log(1/\delta)$,  then with probability at least $1-\delta$, that
\begin{equation}
u^\top M v \leq 2 \sqrt{N\log(1/\delta)}.
\label{eq:uv-bound}
\end{equation}

We now use this observation in conjunction with a \emph{covering argument} to bound $\twonorm{M}$.
\begin{proposition}
Let $x_i \in \R^{n}$ and $w_i\in \R^m$ be such that $x_i\iid\Normal(0,\Sigma_x)$ and $w_i \iid\Normal(0,\Sigma_w)$, and let $M= \sum_{i=1}^N x_i w_i^\top$. Fix a failure probability $\delta \in (0,1]$, and let $N\geq \frac{1}{2}(n+m)\log(9/\delta)$.  Then, it holds with probability at least $1-\delta$ that
\begin{equation}
\twonorm{M} \leq 4 \twonorm{\Sigma_x}^{1/2}\twonorm{\Sigma_w}^{1/2}\sqrt{N(n+m)\log(9/\delta)}.
\end{equation}
\label{prop:2norm-bound}
\end{proposition}
\begin{proof}
First notice that 
\[
M \dist \Sigma_x^{1/2} \left(\sum_{i=1}^N y_iz_i^\top \right) \Sigma_w^{1/2},
\]
for $y_i \iid \Normal(0,I_n)$ and $z_i \iid \Normal(0,I_m)$.  We therefore have that
\[
\twonorm{M} \leq \twonorm{\Sigma_x}^{1/2}\twonorm{\Sigma_w}^{1/2}\twonorm{\sum_{i=1}^N y_iz_i^\top},
\]
thus it suffices to control the term
\begin{equation}
\bigtwonorm{\sum_{i=1}^N y_iz_i^\top} = \sup_{u\in \Sp^{n-1}, v\in \Sp^{m-1}}\sum_{i=1}^N (u^\top y_i)(z_i^\top v).
\end{equation}

We approximate this supremum with an $\epsilon$-net.  In particular, let $\{u_k\}_{k=1}^{M_\epsilon}$ and $\{v_\ell\}_{\ell=1}^{N_\epsilon}$, be $\epsilon$-coverings of the $\Sp^{n-1}$ and $\Sp^{m-1}$, respectively.  Then for every $(u,v) \in \Sp^{n-1}\times\Sp^{m-1}$, let $(u_k,v_\ell)$ denote the elements of their respective nets such that ${\twonorm{u-u_k}\leq \epsilon}$ and $\twonorm{v-v_\ell}\leq \epsilon$.  Then, for an arbitrary matrix $M \in \R^{n\times m}$, we have that
\begin{align*}
u^\top M v &= (u-u_k)^\top M v + u_k^\top M (v-v_\ell) + u_k^\top M v_\ell \leq 2\epsilon \twonorm{M} + \max_{k,\ell} u_k^\top M v_\ell.
\end{align*}
Taking the supremum over $(u,v)$ then shows that
\begin{equation}
\twonorm{M} \leq \frac{1}{1-2\epsilon} \max_{k,\ell} u_k^\top M v_\ell.
\end{equation}
Choosing $\epsilon = 1/4$, a standard volume comparison shows that $M_\epsilon \leq 9^n$ and $N_\epsilon \leq 9^m$ is sufficient, thus
\begin{equation}
\bigtwonorm{\sum_{i=1}^N y_i z_i^\top} \leq 2 \max_{ {1\leq k \leq 9^n},{1 \leq \ell \leq 9^m}}\sum_{i=1}^N (u_k^\top y_i)(z_i^\top v_\ell).
\end{equation}

However, from equation \eqref{eq:uv-bound}, we have that for each pair $(u_k,v_\ell)$, if $N \geq \tfrac{1}{2}(n+m)\log(9/\delta)$, it holds with probability at least $1-\delta/9^{n+m}$ that
\begin{equation}
\sum_{i=1}^N (u_k^\top y_i)(z_i^\top v_\ell) \leq 2 \sqrt{N(n+m)\log(9/\delta)}.
\end{equation}
Union bounding over the $9^{n+m}$ pairs $(u_k,v_\ell)$ proves the claim.
\end{proof}

We now use Proposition \ref{prop:2norm-bound} and a similar argument to lower bound the minimum singular value of a positive definite covariance like matrix.  We note that more sophisticated arguments lead to tighter bounds (e.g., as in \cite{wainwright2019high,tropp2012user}).

\begin{proposition}
Let $x_i \in \R^n$ be drawn i.i.d. from $\Normal(0,\Sigma_x)$, and set $M = \sum_{i=1}^N x_i x_i^\top.$  Fix a failure probability $\delta \in (0,1]$, and let $N \geq 24n\log(9/\delta)$.  Then with probability at least $1-2\delta$, we have that
$\lambda_{\min}(M) \geq \lambda_{\min}(\Sigma_x){N}/{2}.$
\label{prop:lambda_min}
\end{proposition}
\begin{proof}
First notice that $\lambda_{\min}(M) \geq \lambda_{\min}(\Sigma_x)\lambda_{\min}(Z)$, for $Z = \sum_{i=1}^N z_iz_i^\top$, $z_i \iid\Normal(0,I_n)$, and thus it suffices to lower bound the minimum eigenvalue of $Z$.

Let $\{u_k\}$ be a $1/4$-net of $\Sp^{n-1}$ with cardinality at most $9^n$ (note that by symmetry of $Z$, we only require the one $\epsilon$-net).  A similar argument as in the previous proof reveals that
\begin{equation}
\lambda_{\min}(Z) \geq \min_k u_k^\top Z u_k - \twonorm{Z}.
\end{equation}
Since $N\geq \tfrac{1}{2}n\log(9/\delta)$, we have $\twonorm{Z} \leq 4 \sqrt{Nn\log(9/\delta)}$ with probability at least $1-\delta$ by Proposition \ref{prop:2norm-bound}.

Also, we can leverage Proposition \ref{prop:scalar-iid-lb} to show that, for a fixed $u_k \in \Sp^{n-1}$, it holds with probability at least $1-\delta/9^n$ that $u_k^\top Z u_k \geq (N - \sqrt{8Nn\log(9/\delta)}).$

Thus, union bounding over all of the above events, we have with probability at least $1-2\delta$ that
\begin{align*}
\frac{\lambda_{\min}(M)}{\lambda_{\min}(\Sigma_x)} & \geq \left(N - \sqrt{8Nn\log(9/\delta)} - 4\sqrt{Nn\log(9/\delta)}\right).
\end{align*}
The result then follows from the assumed lower bound on $N$.
\end{proof}

Remarkably, we see that the tools developed for the scalar case, suitably augmented with some covering arguments, are all that we needed to derive the aforementioned bounds.  We now show how these two propositions can be used to provide high-probability bounds on $\twonorm{E_A}$ and $\twonorm{E_B}$.

\begin{theorem}
Consider the least-squares estimator defined by \eqref{eq:matrix-ls}.  Fix a failure probability $\delta \in (0,1]$, and assume that $N \geq 24(n_x + n_u)\log(54/\delta)$.  Then, it holds with probability at least $1-\delta$, that
\begin{align}
\twonorm{\Ahat - A} \leq {8\sigma_w}{{\lambda^{-1/2}_{\min}\left(\Sigma_x\right)}}\sqrt{\frac{(2n_x+n_u)\log(54/\delta)}{N}},
\label{eq:Aerror}
\end{align}
and
\begin{equation}
\twonorm{\Bhat - B} \leq \frac{8\sigma_w}{\sigma_u}\sqrt{\frac{(2n_x + n_u)\log(54/\delta)}{N}}.
\label{eq:Berror}
\end{equation}
\label{thm:iid-matrix-error}
\end{theorem}
\begin{proof}
Recalling the expression \eqref{eq:error_A} for $\norm{E_A}_2$, we have by Proposition \ref{prop:2norm-bound}, probability at least $1-\delta/6$ that
\[
\twonorm{Y_N^\top W_N} \leq {4\sigma_w}\sqrt{N(2n_x+n_u)\log(54/\delta)},
\]
and by Proposition \ref{prop:lambda_min}, we have with probability at least $1-2\delta/6$ that $\lambda_{\min}(Y_N^\top Y_N)\geq N/2$.  Union bounding over these events, we therefore have that bound \eqref{eq:Aerror} holds with probability at least $1-\delta/2$.  A similar argument applied to \eqref{eq:error_B} shows that bound \eqref{eq:Berror} holds with probability at least $1-\delta/2$.  Union bounding over these two events yields the claim.
\end{proof}

\section{Single Trajectory Results}
\label{sec:single}

The previous sections made a very strong simplifying assumption: that all of the covariates used in the system-identification step were independent.  To satisfy this assumption, we needed several independent trajectories, from which we only used two-data points.  This is both impractical and data-inefficient -- however, analyzing single trajectory estimators is much more challenging, and is a current active area of research.  This section aims to provide the reader with a survey of some of the tools being used to extend the ideas discussed above to the single-trajectory setting.


\subsubsection{Linear Response}

In this section we study single trajectory results for LTI systems.
We will frame the problem in a more general setting from \cite{simchowitz2018learning}
and specialize the results
to LTI systems.
Suppose that $\{ z_t \} \subseteq \R^n$ is a stochastic process.
Suppose we observe $\{z_t\}$ and the following linear responses $\{y_t\} \subseteq \R^\ell$, defined as:
\begin{align}
	y_t = \Theta_\star z_t + w_t \:, \label{eq:linear_response_model}
\end{align}
where $\Theta_\star \in \R^{\ell \times n}$ is an unknown parameter that we wish to identify and
we assume that $w_t | \calF_{t-1}$ is a zero-mean $\sigma_w$-sub-Gaussian
random vector, where $\calF_t = \sigma(w_0, ..., w_t, z_1, ..., z_t)$. We are interested in the quality of the estimate:
\begin{align}
	\Thetah = \arg\min_{\Theta \in \R^{\ell \times n}} \frac{1}{2} \sum_{t=1}^{T} \norm{y_t - \Theta x_t}_2^2 \:. \label{eq:linear_response}
\end{align}
Notice that \eqref{eq:linear_response} covers the case of 
an autonomous LTI system $x_{t+1} = A x_t + w_t$ where we want to learn the parameter $A$
by setting $y_{t} = x_{t+1}$. It also covers the case of 
a controlled LTI system $x_{t+1} = A x_t + B u_t + w_t$ where we want to learn 
$\Theta = \begin{bmatrix} A & B \end{bmatrix}$.
Now, under the necessary invertibility assumptions, the error $\Thetah - \Theta_\star$ is given by
the expression:
\begin{align*}
	\Thetah - \Theta_\star = \left( \sum_{t=1}^{T} w_t z_t^\top \right) \left( \sum_{t=1}^{T} z_tz_t^\top \right)^{-1} \:.
\end{align*}
We analyze the error by
bounding $\norm{\Thetah - \Theta_\star}$ by the following decomposition:
\begin{align}
	&\bignorm{\left( \sum_{t=1}^{T} w_t z_t^\top \right) \left( \sum_{t=1}^{T} z_tz_t^\top \right)^{-1/2} \left( \sum_{t=1}^{T} z_tz_t^\top \right)^{-1/2}} \nonumber \\
	&\leq \bignorm{ \left( \sum_{t=1}^{T} w_t z_t^\top \right) \left( \sum_{t=1}^{T} z_tz_t^\top \right)^{-1/2} } \bignorm{  \left( \sum_{t=1}^{T} z_tz_t^\top \right)^{-1/2} } \nonumber \\
	&= \frac{  \bignorm{ \left( \sum_{t=1}^{T} w_t z_t^\top \right) \left( \sum_{t=1}^{T} z_tz_t^\top \right)^{-1/2} } }{ \sqrt{ \lambda_{\min}(    \sum_{t=1}^{T} z_tz_t^\top   )     }   } \:. \label{eq:OLS_decomposition}
\end{align}
The term appearing in the numerator of \eqref{eq:OLS_decomposition}
is a \emph{self-normalized martingale} (see e.g. \cite{delapena2009book,abbasi11b}).
On the other hand, the term appearing in the denominator of \eqref{eq:OLS_decomposition}
is the minimum eigenvalue of the empirical covariance matrix.
The analysis of \eqref{eq:OLS_decomposition} proceeds by upper bounding
the martingale term and lower bounding the minimum eigenvalue.
We note that the martingale term can be handled with the self-normalized
inequality from Theorem 1 of Abbasi-Yadkori et al.~\cite{abbasi11b} (see Theorem \ref{thm:self-normalized} below).
We will focus on controlling the minimum eigenvalue. 

Before we present (a simplified version of) the technique used in Simchowitz et al.~\cite{simchowitz2018learning},
we discuss a first attempt at controlling the minimum eigenvalue. 
One could in principle leverage the results of the previous subsection by appealing to mixing time arguments
(see e.g. \cite{yu94}) which allow us to treat the process $\{z_t\}$ as nearly independent across time
by arguing that long term dependencies do not matter. 
However, such arguments yield bounds that degrade as 
the system mixes slower. For the LTI case, this leads to bounds that degrade as the
spectral radius of $A$ approaches one (and is not applicable to unstable $A$).  Instead, we will present the small-ball style of argument used in Simchowitz et al.~\cite{simchowitz2018learning}.
\begin{definition}
\label{def:bmsb}
Suppose that $\{\phi_t\}$ is a real-valued stochastic process adapted to the filtration $\{\calF_t\}$.
We say the process $\{\phi_t\}$ satisfies the $(k, \nu, p)$ block martingale small-ball (BMSB) condition if:
\begin{align}
	\frac{1}{k} \sum_{i=1}^{k} \Pr(\abs{\phi_{t+i}} \geq \nu | \calF_t) \geq p \:\:\text{a.s. for all}\:\: t \geq 0 \:.
\end{align}
\end{definition}

We utilize Definition~\ref{def:bmsb} in the following manner. Recall that we can write
the minimum eigenvalue of the covariance matrix as the following empirical process:
\begin{align*}
	\lambda_{\min}\left( \sum_{t=1}^{T} z_tz_t^\top \right) = \inf_{v \in \R^n : \norm{v}_2 = 1} \sum_{t=1}^{T} \ip{z_t}{v}^2 \:.	
\end{align*}
For a fixed unit vector $v$, we use Definition~\ref{def:bmsb} applied to the process 
$\phi_t = \ip{z_t}{v}$ to lower bound the 
quantity $\sum_{t=1}^{T} \ip{z_t}{v}^2$. We then appeal to a simple covering argument 
to pass to the infimum.
The following proposition allows us to obtain a pointwise bound for 
$\sum_{t=1}^{T} \ip{z_t}{v}^2$.
\begin{proposition}
\label{prop:small_ball_concentration}
Suppose the process $\{\phi_t\}$ satisfies the $(k, \nu, p)$ block martingale small-ball condition. Then,
\begin{align*}
	\Pr\left(  \sum_{t=1}^{T} \phi_t^2 \leq \frac{\nu^2 p^2}{8} k \floor{T/k} \right) \leq \exp{-\frac{\floor{T/k}p^2}{8}} \:.
\end{align*}
\end{proposition}
The remaining covering argument is conceptually simple.
We start by defining the matrix $Z \in \R^{T \times n}$ where the $t$-th row of $Z$ is $z_t$.
Next we fix $(m, M)$ satisfying $0 < m \leq M$
and we set $\varepsilon = m/(4M)$. 
We let $\calN(\varepsilon)$ be a $\varepsilon$-net of the sphere $\calS^{n-1}$,
and consider the two events $\calE_{\mathrm{min}}, \calE_{\mathrm{max}}$ defined as:
\begin{align*}
	\calE_{\mathrm{min}} &= \bigcap_{v \in \calN(\varepsilon)} \{ v^\top Z^\top Z v \geq m \},\,
	\calE_{\mathrm{max}} = \{  \norm{Z^\top Z} \leq M \} \:.
\end{align*}
Note that we can upper bound
$\abs{\calN(\varepsilon)} \leq (1+2/\varepsilon)^n \leq (9\frac{M}{m})^n$.

Then we have:
\begin{align*}
	&\lambda_{\min}\left( Z^\top Z \ind_{\calE_{\mathrm{min}} \cap \calE_{\mathrm{max}}} \right) 
	\geq \inf_{\norm{v}_2 = 1} v^\top Z^\top Z v \ind_{\calE_{\mathrm{min}} \cap \calE_{\mathrm{max}}} \\
	&\geq \min_{v \in \calN(\varepsilon)} \sum_{t=1}^{T} \ip{v}{z_t}^2 \ind_{\calE_{\mathrm{min}} \cap \calE_{\mathrm{max}}} - 2 \bignorm{  Z^\top Z \ind_{\calE_{\mathrm{min}} \cap \calE_{\mathrm{max}}} } \varepsilon \\
	&\geq m - 2 M \varepsilon \geq m/2 \:.
\end{align*}

Proposition~\ref{prop:small_ball_concentration}
combined with a union bound allows us to choose an $m$ such that $\Pr(\calE_{\mathrm{max}}^c) \leq \delta/2$.
In particular, suppose that for every $v \in \calS^{n-1}$ we have that $\phi_t = \ip{z_t}{v}$ 
satisfies $(k, \nu, p)$ BMSB. Then Proposition~\ref{prop:small_ball_concentration} tells us that
if we set $m = \frac{\nu^2 p^2}{8} k \floor{T/k}$ and
if 
\begin{align*}
	\floor{T/k} \geq \frac{8}{p^2} ( \log(2/\delta) + n \log(9 M/m) ) \:,
\end{align*}
then $\Pr(\calE_{\mathrm{min}}^c) \leq \delta/2$.
On the other hand, we can use Markov's inequality to choose an $M$ such that
$\Pr(\calE_{\mathrm{max}}^c) \leq \delta/2$. In particular:
\begin{align*}
	&\Pr( \norm{Z^\top Z} \geq t) \leq t^{-1} \E[\norm{Z^\top Z}] \leq t^{-1} \Tr(\E[Z^\top Z]) = t^{-1} \sum_{t=1}^{T} \E[\norm{z_t}_2^2] \:.
\end{align*}
Therefore we can set $M = 2 \sum_{t=1}^{T} \E[\norm{z_t}_2^2] / \delta$.
Combining these calculations, we see that with probability at least $1-\delta$,
as long as $T$ satisfies:
\begin{align*}
	\floor{T/k} \geq \frac{8}{p^2} \left( \log(\tfrac{2}{\delta}) + n \log\left(  \frac{144}{p^2\delta} \frac{1}{\nu^2 (T-k)} \sum_{t=1}^{T} \E[\norm{z_t}_2^2]     \right) \right), 
\end{align*}
then we have 
$\lambda_{\min}(Z^\top Z) \geq \frac{\nu^2 p^2}{16} k \floor{T/k}$.
As we will see shortly, the
quantity $\frac{1}{\nu^2 (T-k)} \sum_{t=1}^{T} \E[\norm{z_t}_2^2]$
behaves like a condition number for this problem. Fortunately, its contribution
is in the logarithm of the bound, which is a feature of the particular
decomposition in \eqref{eq:OLS_decomposition}.

Combining this technique with bounds on the self-normalized martingale term,
one can show the following result for estimation in the linear response model \eqref{eq:linear_response_model}.
\begin{theorem}
\label{thm:LWM_main}
Fix $\Gamma_{\min},\Gamma_{\max}$ such that $0 \prec \Gamma_{\min} \preceq \Gamma_{\max}$.
Put $\overline{\Gamma} = \Gamma_{\max} \Gamma_{\min}^{-1}$.
Suppose for every fixed $v \in \calS^{n-1}$, we have that (i)
$\phi_t = \ip{z_t}{v}$ satisfies the $(k, \sqrt{v^\top \Gamma_{\min} v}, p)$ block martingale small-ball
condition, and also suppose (ii) that
$\Pr( \sum_{t=1}^{T} z_tz_t^\top \not\preceq T \cdot \Gamma_{\max}) \leq \delta$.
Then as long as $T$ satisfies:
\begin{align*}
	T \geq \frac{10k}{p^2} \left( \log(1/\delta) + 2 n \log(10/p) + \log\det(\overline{\Gamma}) \right) \:,
\end{align*}
we have with probability at least $1-3\delta$,
\begin{align*}
	\norm{\Thetah - \Theta_\star} \leq \frac{90\sigma_w}{p} \sqrt{\frac{\ell + n \log(\frac{10}{p}) + \log\det(\overline{\Gamma}) + \log(\frac{1}{\delta})  }{T \lambda_{\min}(\Gamma_{\min})}}   \:.
\end{align*}
\end{theorem}
We now show how to use Theorem~\ref{thm:LWM_main}
to study the estimation of LTI systems. We consider the simple autonomous case
of $x_{t+1} = A x_t + w_t$ where $w_t \iid \calN(0, \sigma_w^2 I)$.
Here, we will assume that $A$ is stable (i.e. $\rho(A) \leq 1$).
For the autonomous LTI system, we have that $z_t = x_t$, so the $(k, \nu, p)$ BMSB
condition is equivalent to:
\begin{align*}
	\frac{1}{k} \sum_{i=1}^{k} \Pr(\abs{\ip{x_{t+i}}{v}} \geq \nu | \calF_t) \geq p \;.
\end{align*}
We first observe that for $i \geq 1$,
$x_{t+i} | \calF_t \dist \calN( A^i x_t , \Gamma_i )$,
where $\Gamma_i := \sigma_w^2 \sum_{k=0}^{i-1} A^k (A^k)^\T$.
Therefore by a Paley-Zygmund type inequality, 
we have 
$\Pr( \abs{\ip{x_{t+i}}{v}} \geq \sqrt{ v^\top \Gamma_i v } ) \geq 3/10$.
Now fix any $k \geq 1$ and let $1 \leq i_0 \leq k$.
Observe that:
\begin{multline*}
	\frac{1}{k} \sum_{i=1}^{k} \Pr(\abs{\ip{x_{t+i}}{v}} \geq \sqrt{v^\top \Gamma_{i_0} v} | \calF_t) \geq \frac{1}{k} \sum_{i=i_0}^{k} \Pr(\abs{\ip{x_{t+i}}{v}} \geq \sqrt{v^\top \Gamma_{i_0} v} | \calF_t) \\
	\geq \frac{1}{k} \sum_{i=i_0}^{k} \Pr(\abs{\ip{x_{t+i}}{v}} \geq \sqrt{v^\top \Gamma_{i} v} | \calF_t)\geq \frac{3}{10} \frac{k - i_0 + 1}{k}
\end{multline*}
Now we select $i_0 = \ceil{k/2}$.
This shows that $\ip{x_t}{v}$ satisfies the
$(k, \sqrt{v^\top \Gamma_{\ceil{k/2}} v}, 3/20)$ BMSB condition.

We now check condition (ii) of the hypothesis of Theorem~\ref{thm:LWM_main}.
First, we observe that $\E[ \sum_{t=1}^{T} x_tx_t^\top ] = \sum_{t=1}^{T} \Gamma_t \preceq T \cdot \Gamma_T$.
Markov's inequality implies that
$\Pr( \sum_{t=1}^{T} x_tx_t^\top \not\preceq \frac{nT}{\delta} \cdot \Gamma_T) \leq \delta$.
Hence we can set $\Gamma_{\max} = \frac{n}{\delta} \cdot \Gamma_T$.
Theorem~\ref{thm:LWM_main}
tells us that if $T \gtrsim k (n \log(n/\delta) + \log\det(\Gamma_T \Gamma_{\ceil{k/2}}^{-1})) \:,$
then with probability at least $1-\delta$,
\begin{align*}
	\norm{\Ah - A} \lesssim \sigma_w \sqrt{\frac{  n\log(n/\delta) + \log\det(\Gamma_T \Gamma_{\ceil{k/2}}^{-1})   }{T \lambda_{\min}( \Gamma_{\ceil{k/2}} )  } } \:.
\end{align*}

We now discuss how to choose the free parameter $k$,
which depends on whether or not the system $A$ is 
strictly stable $\rho(A) < 1$ or only
marginally stable $\rho(A) = 1$.

\paragraph{Case $A$ is strictly stable}

When $A$ is strictly stable, we have that $\Gamma_\infty = \lim_{t \to \infty} \Gamma_t$ exists and furthermore,
there exists a $\tau \geq 1$ and $\rho \in (0, 1)$ such that
$\norm{A^k} \leq \tau \rho^k$ for all $k=0, 1, ...$.
We can furthermore bound $\norm{\Gamma_\infty - \Gamma_t} \leq  \frac{\sigma_w^2 \tau^2}{1-\rho^2} \rho^{2t}$.
This means that if we choose $k = \frac{1}{1-\rho} \log\left( \frac{2\sigma_w^2 \tau^2}{1-\rho^2}    \right)$, then we have
$\lambda_{\min}( \Gamma_{\ceil{k/2}} ) \geq \lambda_{\min}(\Gamma_\infty)/2$
and also $\log\det(\Gamma_\infty \Gamma_{\ceil{k/2}}^{-1}) \leq n$.
Therefore, Theorem~\ref{thm:LWM_main} implies that
if $T \gtrsim \frac{n\log(n/\delta)}{1-\rho} \log\left( \frac{2\sigma_w^2 \tau^2}{1-\rho^2} \right)$,
then with probability at least $1-\delta$ we have
$\norm{\Ah - A} \lesssim \sigma_w \sqrt{\frac{  n\log(n/\delta) }{  T \lambda_{\min}(\Gamma_\infty)  }}$.

\paragraph{Case $A$ is marginally stable}

This case is more delicate. It is possible to give a general rate that depends on 
various properties of the Jordan blocks of $A$, as is done in Corollary
A.2 of \cite{simchowitz2018learning}. Here, we present a special case when $A$ is an orthogonal matrix
In this case, $\Gamma_t = \sigma_w^2 t \cdot I$. Hence if we set
$k = T / (n \log(n/\delta)$,
then if $T \gtrsim n \log(n/\delta)$, we have that
$\norm{\Ah - A} \lesssim \frac{n \log(n/\delta)}{T}$.
Observe that the rate in this case is actually the faster $O(1/T)$ rate instead of 
$O(1/\sqrt{T})$ when $A$ is strictly stable.

\section{Data-Dependent Bounds and the Bootstrap}
\label{sec:data-dependent}

The previous results characterize upper bounds on the rates of convergence of ordinary least-squares based estimates of system parameters.  Although informative from a theoretical perspective, they cannot be used to implement control algorithms as they depend on properties of the true underlying system.  In this section, we present two data-dependent approaches to computing error estimates.

We begin with the multiple-trajectory independent data setting.  The following proposition from \cite{learning-lqr} provides refined confidence sets on the estimates $(\Ahat,\Bhat)$.

\begin{proposition}[Proposition 2.4, \cite{learning-lqr}]
\label{prop:data_dependent}
Assume we have $N$ independent samples $(y^{(\ell)}, x^{(\ell)}, u^{(\ell)})$ such that
\begin{align*}
y^{(\ell)} = A x^{(\ell)} + B u^{(\ell)} + w^{(\ell)},
\end{align*}
where $w^{(\ell)}\iid\mathcal{N}(0,\sigma_w^2 I_\statedim)$ and are independent from $x^{(\ell)}$ and $u^{(\ell)}$. Also, assume that $N \geq \statedim + \inputdim$.
Then, with probability $1 - \delta$, we have
\begin{align*}
\begin{bmatrix}
(\widehat{A} - A)^\top \\
(\widehat{B} - B)^\top
\end{bmatrix}
\begin{bmatrix}
(\widehat{A} - A)^\top \\
(\widehat{B} - B)^\top
\end{bmatrix}^\top
\preceq C^2_{\statedim, \inputdim, \delta} \left(\sum_{\ell = 1}^N
\begin{bmatrix}
x^{(\ell)}\\
u^{(\ell)}
\end{bmatrix}
\begin{bmatrix}
x^{(\ell)}\\
u^{(\ell)}
\end{bmatrix}^\top
\right)^{-1},
\end{align*}
where $C^2_{\statedim, \inputdim, \delta} = \sigma_w^2 (\sqrt{\statedim + \inputdim} + \sqrt{\statedim} + \sqrt{2 \log(1/\delta)})^{2}$.
If the matrix on the right hand side has zero as an eigenvalue, we define the inverse of that eigenvalue to be infinity.
\end{proposition}

\begin{proof}
We rely on the the following standard result in high-dimensional statistics \cite{wainwright2019high}: for $W \in \R^{N \times \statedim}$ a matrix with each entry i.i.d. $\Ncal(0, \sigma_w^2)$, it holds with probability at least $1 - \delta$ that $\|W\|_2 \leq \sigma_w(\sqrt{N} + \sqrt{\statedim} + \sqrt{2 \log(1/\delta)}).$

As before we use $Z$ to denote the $N \times (\statedim + \inputdim)$ matrix with rows equal to $z_\ell^\top = \begin{bmatrix}
(x^{(\ell)})^\top & (u^{(\ell)})^\top
\end{bmatrix}$. Also, we denote by $W$ the $N \times \statedim$ matrix with columns equal to $w^{(\ell)}$. Therefore, as before, the error matrix for the ordinary least squares estimator satisfies $E = (Z^\top Z)^{-1} Z^\top W$
when the matrix $Z$ has rank $\statedim + \inputdim$. Under the assumption that $N \geq \statedim + \inputdim$ we consider the singular value decomposition  $Z = U \Lambda V^\top$, where $V, \Lambda \in \R^{(\statedim + \inputdim) \times (\statedim + \inputdim) }$ and $U \in \R^{N \times (\statedim + \inputdim)}$. Therefore, when $\Lambda$ is invertible, $E = V (\Lambda^\top \Lambda)^{-1} \Lambda^\top U^\top W = V \Lambda^{-1} U^\top W.$

This implies that
\begin{align*}
E E^\top &= V \Lambda^{-1} U^\top W W^\top U \Lambda^{-1} V^\top \preceq \| U^\top W\|_2^2 V \Lambda^{-2} V^\top= \| U^\top W\|_2^2 (Z^\top Z)^{-1}.
\end{align*}

As the columns of $U$ are orthonormal, the entries of $U^\top W$ are i.i.d. $\Ncal(0, \sigma_w^2)$, from which the result follows.
\end{proof}

The following corollary is then immediate.

\begin{coro}
Let the conditions of Proposition~\ref{prop:data_dependent} hold.  Then with probability at least $1-\delta$
\begin{equation}
\begin{array}{rcl}
\norm{\Ahat-A}_2 &\leq& C(\statedim, \inputdim, \delta) \norm{Q_A MQ_A^\top}^{1/2}_2 \\
\norm{\Bhat-B}_2 &\leq& C(\statedim, \inputdim, \delta) \norm{Q_B MQ_B^\top}^{1/2}_2
\end{array}
\label{eq:mult-traj-bounds}
\end{equation}
\end{coro}

\paragraph*{Single trajectory bounds} To derive similar data-dependent bounds for the single-trajectory setting, we exploit the decomposition \eqref{eq:OLS_decomposition} and find a data-dependent bound to the self-normalized martingale term
\begin{equation}
\norm{(\sum_{t=1}^T z_t z_t^\top)^{-1/2}(\sum_{t=1}^T z_tw_t^\top)}_2,
\label{eq:snm-term}
\end{equation}
where for convenience, we use the transpose of the previously defined expressions.  To do so, we need the following result.

\begin{theorem}[Theorem 1, \cite{abbasi11b}]\label{thm:self-normalized}
Let $\{\calF_t\}_{t\geq 0}$ be a filtration, and $\{\eta_t\}_{t\geq 0}$ be a real valued stochastic process such that $\eta_t$ is $\calF_t$-measurable, and $\eta_t$ is conditionally sub-Gaussian with parameter $R^2$, i.e., 
\[ \forall \lambda \in \R, \, \E\left[e^{\lambda\eta_t} \, | \, \calF_{t-1}\right] \leq \exp{\frac{\lambda^2 R^2}{2}}.
\]
Let $\{X_t\}_{t\geq 0}$ be an $\R^n$-valued stochastic process such that $X_t$ is $\calF_{t-1}$-measurable.  Assume that $V$ is a $n \times n$ dimensional positive definite matrix.  For any $t \geq 0$, define
\begin{equation}
\bar{V}_t = V + \sum_{s=0}^t X_sX_s^\top, \ S_t = \sum_{s=1}^t \eta_s X_s.
\end{equation}
Then, for any $\delta >0$, with probability at least $1-\delta$ and for all $t \geq 0$, 
\begin{equation}
\norm{V_t^{-1/2}S_t}_2^2 \leq 2R^2\log\left(\frac{\det(\bar{V}_T)^{1/2}\det(V)^{-1/2}}{\delta}\right).
\end{equation}
\end{theorem}

This theorem now allows us to provide a purely data-dependent bound to the term \eqref{eq:snm-term}.  We let $V_T := \sum_{t=1}^T z_tz_t^\top$.

\begin{proposition}\label{prop:data-dependent-single-traj}
Fix $\alpha > 0$, and let \[
V = T\cdot \mathrm{blkdiag}(BB^\top \sigma_u^2 + \sigma_w^2 I_{n_x},\, 
\sigma_u^2I_{n_u}).
\]
Then, if $V_T \succeq \alpha V$, it holds with probability at least $1-\delta$, 
\begin{equation}
\norm{\Thetah - \Theta_\star} \leq \sqrt{8(1+\alpha)}\sigma_w\sqrt{\frac{n_x\log(9/\delta) + \tfrac{1}{2}\log\det (V_TV^{-1})}{\lmin{V_T}}}.
\label{eq:data-dependent-single-traj}
\end{equation}
\end{proposition}
\begin{proof}
By assumption $V_T \succeq \alpha V\succ 0$. Now define
\begin{align*}
 \bar{V}_T &:= V_T + V, \, S_T := \sum_{t=1}^T z_tw_t^\top.
\end{align*}
In terms of these matrices, our goal is now to bound $\norm{V_T^{-1/2}S_T}_2$.  We first observe that $V_T\succeq \alpha V$, and consequently $\bar{V}_T \preceq (1+\alpha) V_T$, from which it follows that $\norm{V_T^{-1/2}S_T}_2\leq \sqrt{1+\alpha}\norm{\bar{V_T}^{-1/2}S_T}_2$.  Now fix an arbitrary $v \in \Sp^{n_x-1}$, and notice that $S_Tv = \sum_{t=1}^T z_t(w_t^\top v)$, where $(w_t^\top v) \iid \Normal(0,\sigma_w^2)$.  We can therefore apply Theorem \ref{thm:self-normalized} to see that with probability at least $1-\delta$,
\begin{equation}
\norm{\bar{V}_T^{-1/2}S_Tv}^2_2 \leq 2\sigma^2_w\log\left(\frac{\det(\bar{V}_T)^{1/2}\det(V)^{-1/2}}{\delta}\right).
\label{eq:snm-fixed-v}
\end{equation}
Now taking a $1/4$-net of $\Sp^{n_x-1}$ and union bounding over the at most $9^{n_x}$ events $\eqref{eq:snm-fixed-v}$ with failure probabilities $\delta/9^{n_x}$, we have that with probability at least $1-\delta$
\begin{equation}
\norm{\bar{V}_T^{-1/2}S_T}^2_2 \leq 8\sigma^2_w\left(n_x\log(9/\delta) +\tfrac{1}{2}\log\det(\bar{V}_T V^{-1})\right).
\label{eq:snm-sphere}
\end{equation}

Recalling that $\norm{V_T^{-1/2}S_T}_2\leq \sqrt{1+\alpha}\norm{\bar{V_T}^{-1/2}S_T}$, we have, with probability at least $1-\delta$, that
\begin{equation}
\norm{V_T^{-1/2}S_T}_2 \leq \sqrt{8(1+\alpha)\sigma^2_w n_x\log(9/\delta) + \tfrac{1}{2}\log\det (V_TV^{-1})}.
\end{equation}
Combining this bound with expression \eqref{eq:OLS_decomposition}, then yields \eqref{eq:data-dependent-single-traj}.
\end{proof}

\paragraph*{The Bootstrap} The Bootstrap~\cite{efron79} is a technique used to estimate population statistics (such as confidence intervals) by sampling from synthetic data generated from empirical estimates of the underlying distribution.  Algorithm~\ref{alg:bootstrap},\footnote{We assume that $\sigma_u$ and $\sigma_w$ are known. Otherwise they
can be estimated from data.}, as suggested in \cite{learning-lqr}, can be used to estimate the error bounds ${\epsilon}_A:= \norm{\Ahat - A}_2$ and ${\epsilon}_B:= \norm{\Bhat - B}_2$.

\begin{center}
\begin{algorithm}[h!]\scriptsize
   \caption{Bootstrap estimation of $\epsilon_A$ and $\epsilon_B$}
\begin{algorithmic}[1]
\STATE \textbf{Input:} confidence parameter $\delta$, number of trials $M$, data $\{(\vecx_{t}^{(i)}, \vecu_t^{(i)})\}_{\substack{1 \leq i \leq N\\ 1 \leq t \leq T}}$, and $(\Ahat, \Bhat)$ a minimizer of
$
\sum_{\ell = 1}^N \sum_{t = 0}^{T - 1} \frac{1}{2}\spectralnorm{A\vecx_{t}^{(\ell)} + B \vecu_{t}^{(\ell)} - \vecx_{t + 1}^{(\ell)}}^2.
$
\FOR{$M$ trials}
\FOR{$\ell$ from $1$ to $N$}
\STATE $\widehat{\vecx}_0^{(\ell)} = \vecx_0^{(\ell)}$
\FOR{$t$ from $0$ to $T - 1$}
\STATE $\widehat{\vecx}_{t + 1}^{(\ell)} = \Ahat \widehat{\vecx}_t^{(\ell)} + \Bhat \widehat{\vecu}_t^{(\ell)} + \widehat{\vecw}_t^{(\ell)}$ with $\widehat{\vecw}_t^{(\ell)} \iid \Ncal(0, \sigma_w^2 I_\statedim)$ and $\widehat{\vecu}_t^{(\ell)} \iid \Ncal(0, \sigma_u^2 I_\inputdim)$.
\ENDFOR
\ENDFOR
\STATE $(\widetilde{A}, \widetilde{B}) \in \arg\min_{(A,B)} \sum_{\ell = 1}^N \sum_{t = 0}^{T - 1} \frac{1}{2} \ltwonorm{ A\widehat{\vecx}_{t}^{(\ell)} + B\widehat{\vecu}_{t}^{(\ell)} - \widehat{\vecx}_{t + 1}^{(\ell)}}^2$.
\STATE record $\widetilde{\epsilon}_A = \ltwonorm{\Ahat - \widetilde{A}}$ and $\widetilde{\epsilon}_B = \ltwonorm{\Bhat - \widetilde{B}}$.
\ENDFOR
\STATE \textbf{Output:} ${\epsilon}_A$ and ${\epsilon}_B$, the $100(1- \delta)$th percentiles of the $\widetilde{\epsilon}_A$'s and the $\widetilde{\epsilon}_B$'s.
\end{algorithmic}
   \label{alg:bootstrap}
 \end{algorithm}
\end{center}

For ${\epsilon}_A$ and ${\epsilon}_B$ estimated by Algorithm~\ref{alg:bootstrap} we intuitively have
\begin{align*}
\PP(\spectralnorm{A - \Ahat} \leq {\epsilon}_A) \approx 1 - \delta,~ \PP(\spectralnorm{B - \Bhat} \leq {\epsilon}_B) \approx 1 - \delta.
\end{align*}

Although we do not do so here, these results can be formalized: for more details see texts by Van Der Vaart and
Wellner~\cite{van1996weak}, Shao and Tu~\cite{shao2012jackknife}, and
Hall~\cite{hall2013bootstrap}.

\begin{example}[Data-driven bounds]
Consider the discrete-time double integrator system
\begin{equation}
x_{t+1} = \begin{bmatrix}1 & 0.1 \\ 0 & 1\end{bmatrix} x_t + \begin{bmatrix}0 \\ 1 \end{bmatrix} u_t + w_t,
\end{equation}
driven by noise process $w_t \iid \Normal(0,\sigma_w^2 I_2)$ and exploratory input $u_t \iid \Normal(0,\sigma_u^2)$, with $\sigma_w = 0.1$ and $\sigma_u = 1$.  Figure \ref{fig:data-dependent-bounds} illustrates the resulting error and bound trajectories.
\begin{figure}
 \includegraphics[width=\columnwidth]{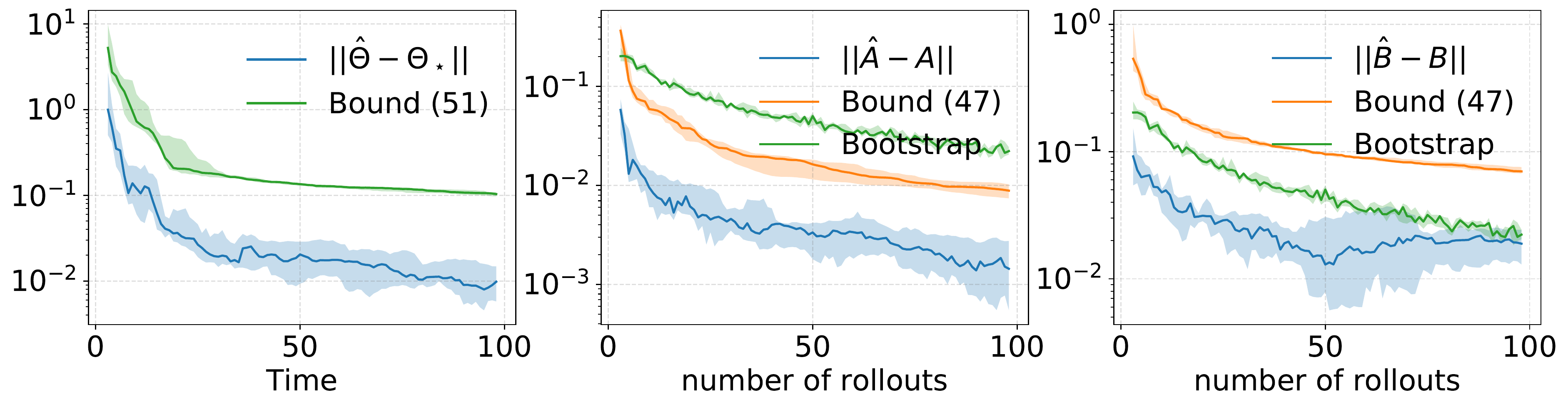}
 \caption{\small Data-dependent and bootstrapped bounds of estimates -- shown are median, first and third quartiles over 10 independent runs.  The left figure shows the error bound \eqref{eq:data-dependent-single-traj} for a single trajectory estimate, whereas the two rightmost figures show the bounds from \eqref{eq:mult-traj-bounds} and bootstrap estimates.  For the bootstrap algorithm, we run $M=200$ loops, and set the horizon $N$ to $T$.} 
 \label{fig:data-dependent-bounds}
\end{figure}
\end{example}

\section{conclusion}
In this paper, we provided a brief introduction to tools useful for the finite-time analysis of system identification algorithms.  We studied the full information setting, and showed how concentration of measure of sub-Gaussian and sub-exponential random variables are sufficient to analyze the independent trajectory estimator.  We further showed that the analysis becomes much more challenging in the single-trajectory setting, but that tools from self-normalized martingale theory and small-ball probability are useful in this context.  Finally, we provided computable data-dependent bounds that can be used in practical algorithms.  In our companion paper \cite{extended}, we show how these tools can be used to design and analyze self-tuning and adaptive control methods with finite-data guarantees.  Although we focused on the full information setting, we note that many of the techniques described extend naturally to the partially observed setting \cite{oymak2018non,sarkar2019finite,tsiamis2019finite}.

\bibliographystyle{IEEEtran}
{\bibliography{../learning-abridged.bib,../refsAP.bib}}

\end{document}